\documentclass[12pt,reqno]{amsart}
\usepackage{amsmath,amssymb,amsthm,mathtools,wasysym,calc,verbatim,enumitem,tikz,pgfplots,url,hyperref,mathrsfs,cite,fullpage,bbm}

\usepackage{comment}

\newtheorem{theorem}{Theorem}[section]
\newtheorem{lemma}[theorem]{Lemma}

\newtheorem{observation}[theorem]{Observation}

\newtheorem{fact}[theorem]{Fact}

\theoremstyle{definition}

\theoremstyle{remark}

\newcommand\N{\mathbb{N}}
\newcommand\R{\mathbb{R}}
\newcommand\Z{\mathbb{Z}}

\newcommand\cI{\mathcal{I}}

\newcommand\cF{\mathcal{F}}
\def\Pr{\mathbb{P}}

\newcommand\eps{\varepsilon}

\renewcommand{\leq}{\leqslant}
\renewcommand{\geq}{\geqslant}

\def\eps{\varepsilon}

	\def\E{\mathbb{E}}
	
	\def\R{\mathbb{R}}
		
	\def\Z{\mathbb{Z}}
	\def\N{\mathbb{N}}
	\def\PP{\mathbb{P}}
	\def\1{\mathbbm{1}}

	\def\s{\sigma}
	\def\t{\theta}

	\def\la{\langle}
	\def\ra{\rangle}

	\def\Var{\mathrm{Var}}

	\def\Cov{\mathrm{Cov}}

	\def\EE{\mathbb{E}}

	\def\cE{\mathcal{E}}

	\def\cG{\mathcal{G}}

\begin{document}
\title{Cosine polynomials with few zeros}
\author{Tomas Ju\v{s}kevi\v{c}ius$^{1}$} 
\thanks{$^1$The first named author is supported by the European Social Fund (project No 09.3.3-LMT-K-712-02-0151) under grant agreement with the Research Council of Lithuania (LMTLT)}
\address{Vilnius University, Faculty of Mathematics and Informatics, 24 Naugarduko st., Vilnius, LT-03225, Lithuania.} \email{tomas.juskevicius[at]gmail.com}
\author{Julian Sahasrabudhe}
\address{ University of Cambridge, Department of Pure Mathematics and Mathematical Statistics, Wilberforce Road. Cambridge, CB3 0WA, UK.}
\email{jdrs2[at]cam.ac.uk}

\begin{abstract} In a celebrated paper, Borwein, Erd\'{e}lyi, Ferguson and Lockhart constructed cosine polynomials of the form
\[ f_A(x) = \sum_{a \in A} \cos(ax), \]
with $A\subseteq \N$, $|A|= n$ and as few as $n^{5/6+o(1)}$ zeros in $[0,2\pi]$, thereby disproving an old conjecture of J.E. Littlewood.
Here we give a sharp analysis of their constructions and, as a result, prove that there exist examples with as few as 
$C(n\log n)^{2/3}$ roots.\end{abstract}

\maketitle

\section{Introduction}

In his 1968 monograph, J.E. Littlewood \cite{JELw4} collected many interesting problems on the behavior of polynomials and trigonometric sums with restricted coefficients,
a subject on which he worked extensively throughout his life \cite{HardyLittlewood,JELw1,JEL62,JELw2,JELw3,JEL66,LO38,LO45,LO48}. Here we concern ourselves with Problem 22:
\begin{center}
``If the $a_{j}$'s are all integral and all different, what is the lower bound on the number \\ of real zeros of $\sum_{i=1}^{n}\cos(a_{j}t)$. Possibly $n-1$ or not much less.''\end{center}
In 2008, Borwein, Erd\'{e}lyi, Ferguson and Lockhart \cite{BEFL} showed that there exist examples of cosine polynomials having as few as $O(n^{5/6}\log n)$ roots,
thereby disproving the ``$n-1$ or not much less'' part of the statement and giving a upper bound on the minimum number of zeros of a $\{0,1\}$-cosine polynomial. Despite considerable interest in Littlewood's problem, no progress has been made in the past ten years on the upper bound of \cite{BEFL}. In this paper, we show that there exist $\{0,1\}$-cosine polynomials with as few as $n^{2/3+o(1)}$ roots. To state this result, define $Z(f)$ to be the number of zeros of $f$ in the interval $[0,2\pi]$. 
 
\begin{theorem}\label{thm:main1}
For each $n\in \N$ there exists a set $A \subseteq \N$, with $|A| = n$, for which the cosine polynomial
\[ f_A(x) := \sum_{a \in A} \cos(ax), \]
has  \[ Z(f_A) = O((n\log n)^{2/3}). \] \end{theorem}

Obtaining good \emph{lower bounds} on the minimum number of zeros has also remained a challenging problem. 
In 2007, Borwein and Erd\'{e}lyi \cite{BElowerbounds} showed that the number of roots tends to infinity along certain subsequences of functions. In the following year, Borwein, Erd\'{e}lyi, Ferguson and Lockhart \cite{BEFL} then explicitly conjectured that the number of roots tends to infinity as $n\rightarrow \infty$, in general. This conjecture was independently proved by Erd\'{e}lyi~\cite{TEFit} and Sahasrabudhe \cite{Jules1}. In fact, Sahasrabudhe went further and gave the explicit lower bound
$$ ( \log \log \log n  )^{1/2+o(1)},$$
for the number of roots. While the exponent of $1/2+o(1)$ was later improved to $1+o(1)$ by Erd\'{e}lyi~\cite{trsh}, the gap between the upper and lower bounds remains large.

To prove Theorem~\ref{thm:main1}, we give a precise analysis of a class of random polynomials introduced by the authors of \cite{BEFL}. For $m\geq 0$ define $f$ to be the random polynomial 
\begin{equation}\label{eq:def-f} f(x) = \sum_{k=0}^n \cos(kx) - \sum_{k=0}^{m} \eps_k \cos(kx),\end{equation} where $\eps_0,\ldots,\eps_m \in \{0,1\}$ are independent Bernoulli random variables such that $\Pr(\varepsilon_i=0) = \Pr(\varepsilon_i=1)= 1/2$. This defines a probability measure on the set of $\{0,1\}$-cosine polynomials and we write $\cF_{m,n}$ for the probability space obtained in this way. Our second main result gives a sharp bound on the
expected number of zeros\footnote{Here we understand $f \sim \cF_{n,m}$ to mean $f$ \emph{is sampled} from $\cF_{n,m}$.} of $f \sim \cF_{n,m}$. 

\begin{theorem} \label{thm:main2} Let $m,n \in \N$ satisfy $m\leq n$. If $f \sim \cF_{n,m}$ then
\[ \EE\, Z(f) = \Theta\left(\frac{n \log m}{\sqrt{m}}  + m\right). \] \end{theorem}

\noindent We observe that $(n\log n)^{2/3}$ is the minimum of this expectation if we optimize over $m \geq 0$, thus accounting for the bound obtained in Theorem~\ref{thm:main1}. 

Our main new technical result in this paper (Theorem~\ref{thm:ET-for-special-polys}) is a tool that gives a sharp, \emph{deterministic} bound on the number of zeros in short intervals for polynomials of the special form \eqref{eq:def-f}. The standard tool for establishing upper bounds on the number of roots in short intervals is the following famous theorem\footnote{Here we are just stating the Erd\H{o}s-Tur\'{a}n theorem in our setting. The full theorem is considerably more general.} of Erd\H{o}s and Tur\'an \cite{ErdosTuran}. For this, we let $Z_I(f)$ denote the number of zeros of $f$ in $I \subseteq [0,2\pi]$.

\begin{theorem}\label{thm:ET}
Let $f$ be a $\{0,1\}$-cosine polynomial of degree $n$ and let $I \subseteq [0,2\pi]$ be an interval. Then 
\[ Z_I(f) \leq n|I|/(2\pi) + O(\sqrt{n\log n}). \] 
\end{theorem}

While there are various interesting strengthenings of this theorem \cite{amoroso,totik-varju,erd-TE-oneside,steinerberger}, these results appear to be too weak to improve upon the $n^{5/6+o(1)}$ bound, obtained by the authors of \cite{BEFL}, and it is not even clear that it is \emph{possible} to overcome this barrier, as Theorem~\ref{thm:ET} is known to be sharp, or close to sharp, in many cases. On the other hand, it is conceivable that the work of Blatt \cite{blatt} on simple roots could be adapted to this setting to overcome the $n^{5/6}$ barrier. However, there are additional complications with this approach and, even if these were to be successfully navigated, it appears that it would still fall short of the optimal bound. Our new technical tool allows us to overcome the $n^{5/6}$ barrier and obtain a sharp bound on the number of zeros. Indeed, this result can be viewed a sort of Erd\H{o}s-Tur\'{a}n result for polynomials of the form \eqref{eq:def-f}, which performs much better on short intervals. As an added bonus, our proof of Theorem~\ref{thm:ET-for-special-polys} is entirely elementary, whereas the proof of Theorem~\ref{thm:ET} uses methods from Fourier and complex analysis. For the statement of this result, let us write $D_n(x) = \sum_{k=0}^n \cos kx$.

\begin{theorem}\label{thm:ET-for-special-polys} For $n \in \N$, $\alpha \in (0,1)$, let $f = D_n - g$, where $g$ is a $\{0,1\}$-cosine polynomial with 
$\deg(g) \leq n^{1-\alpha}$. If $I \subseteq [C(\alpha n)^{-1} ,\pi]$ is an interval then 
\[ Z_I(f) \leq  C\alpha^{-1}(n|I|  + 1) , \]
where $C >0$ is an absolute constant.
\end{theorem}

\noindent Note that Theorem~\ref{thm:ET-for-special-polys} does not apply when $\deg(g) = n^{1-o(1)}$ and to deal with this case, we need to develop a slightly more sophisticated version of Theorem~\ref{thm:ET-for-special-polys} (See Lemma~\ref{lem:bigm-short-ints} in Section~\ref{sec:large-m}).

Taken a little more broadly, Theorem~\ref{thm:ET-for-special-polys} can be situated in a group of results and conjectures that say that the roots of polynomials with restricted coefficients can't ``clump up'' too much \cite{Schur,bloch-polya,totik-varju,erd-TE-oneside}. 
One extreme version of this is the maximum \emph{multiplicity} of a polynomial at a point; i.e. when the ``interval'' has length zero. One problem in this domain, perhaps pertinent to our work here,
asks for the maximum multiplicity at $1$ of a degree $n$ $\{\pm 1\}$-polynomial. If we let $m(n)$ be the maximum multiplicity, the best bounds \cite{boyd,borweinQs} to date are  
\[ \log n \leq m(n) \leq c(\log n)^{2-o(1)}.\] In a similar vein, the Tarry-Escott Problem, asks for a lower bound on the minimum number of non-zero terms $t(n)$ in a  polynomial with coefficients in $\{0,\pm 1\}$ and a root at $1$ of multiplicity $n$. Here, the best bounds \cite{borweinQs} to date are of the form 
\[ 2n \leq t(n)\leq n^{2+o(1)}.\]

\subsection{Sketch of the proofs of Theorems~\ref{thm:main1} and \ref{thm:main2}} To prove Theorems~\ref{thm:main1} and~\ref{thm:main2}, we follow \cite{BEFL} and use the special property of the polynomial $D_n$ 
\[ |D_n(x)-1/2|\leq s(x),\] where $s(x) = (2\sin x/2)^{-1}$,  to deduce that if $f(x) = D_n(x) - g(x) = 0$ then $|g(x)-1/2|  \leq s(x)$. This motivates the definition of the following set   
\[ \cE(g) := \{ x \in (0,\pi] : |g(x)-1/2| < s(x) \}, \]
which we think of as the set of all $x$ where $f(x)$ \emph{might} have a zero.

In the case $m < n^{1-\alpha}$, we obtain a close relationship between the number of zeros of $f$ and the measure of the set $\cE(g)$. To do this, we write $\cE(g)$
as a union of $O(m)$ intervals and then apply Theorem~\ref{thm:ET-for-special-polys} to each of these intervals to obtain an upper bound on the number of zeros 
in terms of the measure of the set $\cE(g) \subseteq [0,\pi]$. This is the content of Lemma~\ref{lem:zeros-and-measure}, where we additionally get a similar lower bound on the number of zeros
\begin{equation}\label{eq:into-E(g)}
 n|\cE(g)| - m  \lesssim  Z(f) \lesssim n|\cE(g)| + m.
\end{equation}
Up to this point our results have been entirely deterministic and to finish the proof of the upper bound in Theorem~\ref{thm:main1} and Theorem~\ref{thm:main2} (in the case that $m < n^{1-\alpha}$), we simply
need to show that
\[ \EE\, |\cE(g)| =  \Theta\left( \frac{\log m}{m^{1/2}} \right).\]
This we do in Lemma~\ref{lem:time-in-envelope}. To finish the proof of the lower bound on the expectation $\EE\, Z(f)$ in Theorem~\ref{thm:main2}, we need to additionally show that $f(x)$ has, on average, 
at least $cm$ zeros away from $x = 0$. This is what corrects the ``$-m$'' seen in the lower bound in \eqref{eq:into-E(g)}.

In the case $m >n^{1-\alpha}$, the proof is similar but the analysis becomes more complicated. Here we additionally consider the set 
\[ \cE'_n(g) := \{ x : |g'(x)| \leq 2^7 n s(x) \},\]
and then prove a version of Theorem~\ref{thm:ET-for-special-polys} for intervals that are not entirely contained in $\cE_n'(g)$ (Lemma~\ref{lem:bigm-short-ints}). 
To deal with roots inside of $\cE_n'(g)$, we show that the number of intervals contained in $\cE_n'(g)$ (that is, the number of intervals where roots might ``clump up'') is likely to be small. This is ultimately achieved by appealing to a beautiful anti-concentration result of Hal\'{a}sz \cite{halasz}, which we state as Theorem~\ref{thm:Hal}.

\section{Roots in short intervals}\label{sec:Oneroot}

In this section we prove Lemma~\ref{lem:no-wiggle}, which we will use to show that our polynomial does not have multiple roots 
in an interval of length $\approx 1/n$.

\begin{lemma}\label{lem:no-wiggle} For $\delta = 2^{-14} ,\alpha \in (0,1)$, let $I \subseteq (2^{31} (\alpha n)^{-1} ,\pi]$ be an interval with $|I| = \delta/n$,
let $g$ be a $\{0,1\}$-cosine polynomial with $\deg(g) \leq n^{1-\alpha}$ and let $f = D_n - g$. 
Then $f$ has at most $2\alpha^{-1}$ roots in $I$. \end{lemma}

Theorem~\ref{thm:ET-for-special-polys} is immediate from Lemma~\ref{lem:no-wiggle}: we simply cover the given interval $I$ with intervals of length $2^{-14}/n$ and then apply Lemma~\ref{lem:no-wiggle} to deduce that each of these intervals contains at most $2\alpha^{-1}$ roots.

Now, turning towards the proof of Lemma~\ref{lem:no-wiggle}, we record the following consequence of the mean value theorem, the proof of which is deferred to the appendix.

\begin{lemma}\label{lem:mean-value}
For $\ell \in \N$, let $f$ be a trigonometric polynomial with at least $\ell$ zeros in an interval $I \subseteq [0,2\pi]$
of length $|I| = \eta$. Then 
\begin{equation} \label{eq:mean-value} \sup_{x\in I}|f(x)| \leq \eta^{\ell}\sup_{x \in I}|f^{(\ell)}(x)|. \end{equation}
\end{lemma}

\vspace{4mm}

We shall also need an elementary lemma concerning ``high'' derivatives $D_n^{(r)}$ of the function $D_n$. To prepare for this, we need the following two basic facts. Again, the proofs of these can be found in the appendix.

\begin{fact}\label{fact:B} Set $B(x) = 1/\sin(x/2)$. For $x\in (0,\pi]$ and $r \in \N$, we have that 
\[ |B^{(r)}(x)| \leq (2^{r+3}r!)/x^{r+1}.\]
\end{fact}

\vspace{4mm}

\begin{fact}\label{lem:largeSIN}
For $\delta\in [0,1]$, let $I\subset [0,\pi]$ be an interval of length $\delta/n$. Let $C_r(x) = (d/dx)^r\sin(x)$ and set $T = (n+1/2)$. 
Then there exists a point $x_{0}\in I$ such that 
$$|C_r(Tx_0) |\geq \delta/4.$$
\end{fact}

\noindent We are now able to prove Lemma~\ref{lem:d-kernel}. For this,
we use the formula
\begin{equation} \label{eq:Dn-formula} D_n(x) = \frac{1}{2} + \frac{\sin(Tx)}{2\sin(x/2)}, 
\end{equation} for $x \in (0,2\pi)$, $T = n+1/2$.

\begin{lemma}\label{lem:d-kernel} For $n,r \in \N$, and $\delta \in (0,1)$, let $I \subseteq (  2^{15}r\delta^{-1}/n ,\pi]$ be an interval with \\ $|I| = \delta/n$. 
\begin{enumerate}
\item\label{item:d-kernel1} For all $x \geq 2^{15}r/n$ we have $|D^{(r)}(x)| \leq 2^{6}T^r/x$;
\item\label{item:d-kernel2} There exists $x_0 \in I$ for which 
\[ |D^{(r)}_n(x_0)| \geq 2^{-5}\delta T^r/x_0.\]
\end{enumerate}
\end{lemma}
\begin{proof}
Setting $A(x) = \sin(Tx)/2$ and $B(x) = (\sin(x/2))^{-1}$, we have  
\[ D_n^{(r)}(x) = \sum_{i=0}^r \binom{r}{i}A^{(r-i)}(x)B^{(i)}(x). \]
Now set $C_r(x) = (d/dx)^r\sin x$. By rearranging and applying Fact~\ref{fact:B}, we have 
\[ \left|D_n^{(r)}(x) - \frac{T^rC_r(Tx)}{2\sin x/2} \right|  \leq \sum_{i=1}^r \binom{k}{i}|A^{(r-i)}||B^{(i)}| \leq  \sum_{i=1}^r \binom{k}{i} \frac{T^{r-i} 2^{i+3}i!}{x^{i+1}}=\sum_{i=1}^{r}a_i. \]  
Since $x \geq 2^{15}r/n$, we have $a_{i+1}/a_i\leq 1/2$ and so the last sum is at most $2a_1$. We thus have
\begin{equation}\label{eq:Dk-bound}  \left|D_n^{(r)}(x) - \frac{T^rC_r(Tx)}{2\sin x/2} \right|  \leq 2^5 T^{r-1}x^{-2}. \end{equation}
Now to prove Statement~\ref{item:d-kernel1} in Lemma~\ref{lem:d-kernel}, we use that $1/(\sin(x/2)) \leq 1/x$ in (\ref{eq:Dk-bound}) and obtain
\[ |D_n^{(r)}(x)|\leq T^{r}/x+2^{5}T^{r-1}/x^{2}\leq 2^{6}T^{r}/x, \] 
for $ x\geq 2^{15}r/n$. To prove the Statement~\ref{item:d-kernel2} in Lemma~\ref{lem:d-kernel} we apply Fact~\ref{lem:largeSIN} to obtain $x_0 \in I$ such that 
$|C_r(Tx_{0})| \geq \delta/4 $. Then using (\ref{eq:Dk-bound}) and the elementary inequality $|\sin(x/2)| \leq |x/2|$ together with $x_0\geq 2^{15}/(\delta n)$ we obtain
$$|D^{(r)}(x_0)|\geq \delta T^{r}/(4x_{0})-2^4T^{r-1}/x_{0}^{2}\geq \delta 2^{-5}T^{r}/x_{0},$$ thus completing the proof of Lemma~\ref{lem:d-kernel}. \end{proof}

\vspace{4mm}

We are now ready for the key idea behind Lemma~\ref{lem:no-wiggle}: if $f = D_n - g$ contains many roots in an interval of length $\approx n^{-1}$,
then some high derivative $g^{(t)}$ of $g$ must be impossibly large.

\begin{lemma}\label{lem:two-roots-reux}
For $t \in \N$ and $\delta = 2^{-14}$, let $I \subseteq [ 2^{30}(t+2)/n,\pi]$ be an interval of length $\delta/n$,
let $g$ be a trigonometric polynomial and let $f = D_n - g$.  
If $f$ has at least $t+2$ roots in $I$ then there exists $x_0 \in I$ for which
\begin{equation}\label{eq:unlikely}  \max_{x \in I} |g^{(t)}(x)| +  (\delta/n)^2 \max_{x \in I} |g^{(t+2)}(x)| \geq 2^{-6}\delta n^t/x_0. \end{equation}
\end{lemma}
\begin{proof}
Since $I \subseteq [ 2^{30}(t+2)/n,\pi] $ is an interval of length $\delta/n$, Lemma~\ref{lem:d-kernel} guarantees a point $x_0 \in I$ for which $|D^{(t+2)}(x_0)| \geq 2^{-5}\delta n^{t+2}/x_0$. 
 Let's also define
\[ M_r := \max_{x \in I} |g^{(r)}(x)|,\] for all $r \in \N$.
Now, for all $x \in I$, and all $1\leq r\leq t+2$, we may apply Statement~\ref{item:d-kernel1} in Lemma~\ref{lem:d-kernel} to obtain
\begin{equation}\label{equ:fk} |f^{(r)}(x)| \leq |D^{(r)}_n(x)| + |g^{(r)}(x)| \leq 2^8 n^{r}/x + M_r. \end{equation}
Since $f$ has $t+2$ roots in $I$, $f^{(t)}$ has at least $2$ roots in $I$. Thus, we may apply Lemma~\ref{lem:mean-value} to $f^{(t)}$ and \eqref{equ:fk}
to obtain
\[|f^{(t)}(x)| \leq (\delta/n)^2\max_{x \in I} |f^{(t+2)}(x)| \leq  2^7\delta^2 n^{t}/x + (\delta/n)^2 M_{t+2}. \]
We now specify $x = x_0$ to learn
\[ 2^7\delta^{2}n^t/x_0 + (\delta/n)^2 M_{t+2} \geq  |f^{(t)} (x_0)| = \left| D_n^{(t)}(x_0) - g^{(t)}(x_0)\right| \geq |D^{(t)}_n(x_0)| - |g^{(t)}(x_0)|. \]
By rearranging and using the lower bound $|D^{(t)}(x_0)| \geq \delta 2^{-5} n^t/x_0$, we obtain
\[  M_t +  (\delta/n)^2 M_{t+2} \geq  |D^{(t)}_n(x_0)| - 2^7\delta^{2}n^t/x_0 \geq  2^{-5}\delta n^t/x_0(1 - 2^{12}\delta) \geq \delta 2^{-6}n^t/x_0, \]
as desired.
\end{proof}

\vspace{4mm}

\begin{proof}[Proof of Lemma~\ref{lem:no-wiggle}]
We may assume that $Z_I(f) \geq \alpha^{-1} +2$, otherwise we are done. We write  $t = \alpha^{-1}$ and note that since $f$ has $t+2$ roots in $I$, we may apply Lemma~\ref{lem:two-roots-reux}, to see that 
\begin{equation}\label{eq:final-bound}
   2^{-6}\delta n^t/x_0 \leq \max_{x \in I} |g^{(t)}(x)|  + (\delta/n)^2\max_{x \in I} |g^{(t+2)}(x)|,
\end{equation} since $I$ satisfies $I \subseteq [2^{31}\alpha^{-1}n^{-1},\pi] \subseteq [2^{30}(t+2)n^{-1},\pi]$. On the other hand, for all $r \in \N$, we have the crude bound
\[ \max_{x \in I} |g^{(r)}(x)| \leq \sum_{k=1}^m  k^r\left| \cos kx \right| \leq m^{r+1}. \]
So putting this together with \eqref{eq:final-bound}, we get $\delta n^t/(2x_0) \leq 2m^{t+1}$. But this contradicts $t = \alpha^{-1}$
and $m<n^{1-\alpha}$.\end{proof}

\section{A deterministic bound on zeros in terms of $|\cE(g)|$} \label{sec:zeros-and-measure}

In this section we prove Lemma~\ref{lem:zeros-and-measure}, which provides a deterministic bound on the number of zeros of polynomials of the form $f = D_n-g$, where $\deg(g) \leq n^{1-\alpha}$, in terms of the measure of the set $\cE(g)$.

\begin{lemma}\label{lem:zeros-and-measure}For $n \in \N$, let $f = D_n -g$, where $g$ is a $\{0,1\}$-cosine polynomial with $\deg(g) \leq n^{1-\alpha}$.
We have
\begin{equation} \label{eq:zeros-and-measure}  \frac{n}{2\pi}|\cE(g)| - cm \leq Z(f) \leq C_{\alpha}\left(n|\cE(g)| + m\right) + cn^{0.6} ,\end{equation}
where $C_{\alpha}>0$ is a constant only depending on $\alpha$.
\end{lemma}

We urge the reader to ignore the $n^{0.6}$ term in \eqref{eq:zeros-and-measure}, as it is both immaterial in our application and probably not really necessary.

To prove Lemma~\ref{lem:zeros-and-measure}, we first show that $\cE(g)$ can be written as the sum of at most $O(m)$ intervals, a fact that holds simply by the degree of $g$. Recall that 
\[ \cE(g) = \{ x \in (0,\pi] : |g(x)-1/2| < s(x) \} \] and $s(x) = (2\sin(x/2))^{-1}$.

\begin{lemma}\label{lem:few-crossings}
Let $g$ be a trigonometric polynomial with $\deg(g) \leq m$. Then $\cE(g)$ can be written as the union of at most $8m+4$ intervals.
\end{lemma}
\begin{proof} Note that the number of such intervals is at most the number of solutions to\\ $|g(x)-1/2| = s(x) $.
So if we set $y = x/2$, we want to count solutions to $g(2y) -1/2 = \alpha /(2\sin(y))$, where $\alpha \in \{\pm 1\}$. We may fix $\alpha = 1$, and note the other case is symmetric.
Thus, we simply want to bound the number of zeros of the (non-zero) trigonometric polynomial 
\[ h(y) :=\sin(y)(2g(2y) - 1) - 2,  \] for $y \in [0,2\pi]$. The number of 
zeros of this trig polynomial is at most $2\deg(h)\leq 4m+2$. Doubling this, to account for the case $\alpha = -1$, gives us our bound.\end{proof} 

\vspace{4mm}

We now recall the observation made in the Introduction.

\begin{observation}\label{obs:danger} Let $g$ be a function and let $f = D_n - g$. Then all of the zeros of $f$ are in $\cE(g)$.
\end{observation}
\begin{proof}
If $f(x) = 0 $ then, using the formula \eqref{eq:Dn-formula}, we see that
\[ g(x) = D_n(x) = \frac{1}{2} + \frac{\sin(Tx)}{2\sin(x/2)}\]
and re-arranging yields $|g(x)-1/2| \leq s(x)$, as desired.
\end{proof}

\vspace{4mm}

We also have the following companion to this observation; that $f$ is forced to have quite a few roots when $|g(x)-1/2|\leq s(x)$.

\begin{observation} \label{obs:root-in-intI} Let $f = D_n - g$, where $g$ is a continuous function.  
If $I$ is an interval $I \subseteq \cE(g)$ then $Z_I(f) \geq \left\lfloor \frac{T}{2\pi} |I| \right\rfloor$, where $T = n+1/2$.
\end{observation}
\begin{proof} 
In any sub-interval $J \subseteq I$ of length $2\pi/T$, there exist $x_0,x_1 \in J$, such that $D_n(x_0) =1/2-s(x_0)$ and $D_n(x_1) =1/2+s(x_1)$. Since $g$ is continuous and $|g|<s(x)$ the curves $D_n$ and $-g$ cross, resulting in a zero of $f$. To finish, simply note that there are  $\left\lfloor \frac{T}{2\pi} |I| \right\rfloor$ pairwise disjoint intervals of length $2\pi/T$ in the interval $I$. 
\end{proof}

\vspace{4mm}

We now turn to finish the proof of Lemma~\ref{lem:zeros-and-measure}.

\begin{proof}[Proof of Lemma~\ref{lem:zeros-and-measure}]
Since $f$ is a cosine polynomial, it is symmetric about the origin and periodic with period $2\pi$. It is therefore sufficient to count roots in $[0,\pi]$.
We start by proving the upper bound at \eqref{eq:zeros-and-measure}. First note that by Theorem~\ref{thm:ET} (the Erd\H{o}s-Tur\'{a}n Theorem) there are at most 
$O( n^{0.6} )$ roots in the interval $[0,n^{-0.9}]$; so in what follows we may assume that $x \geq n^{-0.9}$ and, in particular, $x \gg n^{-1}$.

Now, by Lemma~\ref{lem:few-crossings}, $\cE':= [n^{-0.9},\pi] \cap \cE(g)$ can be expressed as the union of $t \leq 9m $ intervals $J_1,\ldots,J_{t}$. 
By Observation~\ref{obs:danger}, all of the zeros of $f$ lie in $\cE(g)$. So applying this observation along with Lemma~\ref{thm:ET-for-special-polys} to each of these intervals,
we have 
\[ Z_{[n^{-0.9},\pi]}(f) = Z_{\cE'}(f) \leq Z_{J_1}(f) + \cdots + Z_{J_t}(f)  \leq C_{\alpha}(n|\cE(g)| +  m),\]
as desired. Putting this together with the bound on $Z_{[0,n^{-0.9}]}(f)$, yields the upper bound in Lemma~\ref{lem:zeros-and-measure}.

For the lower bound, we need only to apply Observation~\ref{obs:root-in-intI} to each of the intervals $J_1,\ldots,J_t$ to obtain
\[ Z(f) \geq \sum_{i=1}^{t} \left\lfloor \frac{T|J_i|}{2\pi} \right\rfloor \geq \frac{n}{2\pi}|\cE(g)| - 9m,\] as desired. \end{proof}

\section{Two probabilistic calculations}

For $m \in \N$, we let $g = g_m$ be the random polynomial
\[ g(x) = \sum_{k=0}^m \eps_k\cos(kx), \] where $\eps_1,\ldots,\eps_m \in \{0,1\}$ are independent Bernoulli random variables such that $\PP(\eps_k = 0) = \PP(\eps_k = 1) = 1/2$ and we let $\cG_m$ denote the corresponding the probability space on cosine polynomials of degree $m$. We now turn to calculate the expected size of $|\cE(g)|$.

\begin{lemma} \label{lem:time-in-envelope}
Let $g \sim \cG_m$. Then
\begin{equation}\label{eq:time-bound}
 \EE\, | \cE(g)| = \Theta\left(\frac{\log m}{\sqrt{m}}\right).
\end{equation}
\end{lemma}

To prove Lemma~\ref{lem:time-in-envelope}, we shall use a version of the classical theorem of Berry and Esseen, which roughly says that the random variable $(g(a),g(b))$
behaves quite a bit like a two-dimensional normal distribution with the same mean and covariance matrix. Recall that if $(X,Y)$ is a random variable taking values in $\R^2$, then its \emph{covariance matrix} is defined by 
\[\Sigma(X)  := \begin{pmatrix}
\Var(X) & \Cov(X,Y) \\
\Cov(X,Y) & \Var(Y)
\end{pmatrix}.\]

\noindent The following theorem can be easily derived from the the main theorem of \cite{Raic}. 
\begin{theorem}\label{thm:BE}
Let $X_1,\ldots,X_n$ be independent random vectors in $\R^2$ such that
$$\sum_{i=1}^n \E\, X_i= (\mu_1,\mu_2) = \mu \quad \text{and}\quad \sum_{i=1}^{n}\Sigma(X_i)= \begin{pmatrix}\s_1^2 & 0  \\ 0 & \s_2^2 \end{pmatrix} = \Sigma.$$
If $Z \sim N(\mu,\Sigma)$ then for all measurable convex sets $C$ we have
\begin{equation}\label{Berry}
\left|\PP\left(\sum_{i=1}^{n}X_i\in C\right)-\PP(Z\in C)\right|\leq c \min\{ \s_1^{-3}, \s_2^{-3}\} \sum_{i=1}^{n}\E\|X_i - \E X_i \|_{2}^{3},
\end{equation}
where $c>0$ is an absolute constant.
\end{theorem}

We shall first use a ``one dimensional version'' of the Theorem~\ref{thm:BE}, which says that if the random variables $X_1,\ldots,X_n$ take values in $\R$, then 
\[ |\PP( X \leq t ) - \PP(Z \leq t)| \leq c\s^{-3} \sum_{i=1}^{n}\E\|X_i - \E X_i \|_{2}^{3}, \]
where $Z \sim N(\mu,\s^2)$, $\s^2 = \sum_{i} \Var(X_i)$ and $\mu = \sum_i \EE X_i$. To derive this statement from Lemma~\ref{thm:BE}, simply apply Lemma~\ref{thm:BE}
to the random variables $\tilde{X}_i := (X_i,X'_i)$, where $X_i,X'_i$ are independent copies of $X_i$ and consider the convex set $C_t := \{ x : x \leq t\}^2$.

Before diving into the proof of Lemma~\ref{lem:time-in-envelope}, let us set out a few of the basic probabilistic quantities in play. As we look to apply
Theorem~\ref{thm:BE} to the random sum $g(x)$, we note that
\[ \EE\, g(x) = \E \sum_{k=0}^{m}\eps_k \cos(kx)=(1/2)D_{m}(x),\] and for $x$ such that\footnote{For a set $S$, we define $d(x,S) := \inf_{s\in S}|x-s|\}$.} $d(x,\pi\Z) \geq 2\pi/m$, we have that $\s^2(x) := \Var(g(x))$ is
\begin{equation}\label{eq:s} \s^2(x) = \sum_{k=0}^{m}\Var (\eps_k \cos(kx))^2=\frac{1}{4}\sum_{k=0}^m \cos^{2}(kx)=2^{-3}(m+1+D_{m}(2x)) \geq m/16,  \end{equation}
where we have applied the inequality $|D_{m}(2x)|\leq m/2$ when $d(x,\pi\Z) \geq 2\pi /m$. Finally, we see that the sum of the third moments of the summands is
\begin{equation} \label{eq:beta} \beta_3(x) := \sum_{k=0}^m \E |(\eps_k-\frac{1}{2}) \cos(kx)|^3\leq m. \end{equation}

\vspace{4mm}

\begin{proof}[Proof of Lemma~\ref{lem:time-in-envelope}]
We look to apply the above Berry-Esseen inequality, in one dimension, to each sum $g(x)$ and to the set $\{y : y\leq t\}$. 
Indeed, if we denote by $Z(x)$ the Gaussian random variable with the same first two moments of $g(x)$, we may apply 
Theorem~\ref{thm:BE} to learn
\begin{equation}\label{eq:BE-app} 
|\mathbb{P}(g(x) \leq t)-\mathbb{P}(Z(x) \leq t)|\leq \beta_3(x)/\s(x)^3 = O(m^{-1/2}), \end{equation}
by using \eqref{eq:s} and \eqref{eq:beta}.
Now, turning to the left-hand-side of \eqref{eq:time-bound}, we express
\begin{equation} \label{eq:int} \EE\, |\cE(g)| 
= \int^{\pi}_{0} \PP\left( |g(x)|< s(x) \right) \, dx = \int^{\pi}_{0} \PP\left( |Z(x)|\leq s(x) \right)\, dx + O(m^{-1/2}). \end{equation} 
Since $s(x)  \geq m/16$, when $x$ satisfies $d(x,\pi\Z) \geq 2\pi/m$, we have
\[ \PP(|Z(x)|< s(x)) = \Theta\left(\frac{s(x)}{\s(x)}\right) = \Theta\left(\frac{1}{x m^{1/2}}\right), \]
for such $x$
And so, by removing the first removing the intervals $J_1 := [0,2\pi/m], J_2 := [\pi-2\pi/m,\pi]$ from the integral, we have
\[ \int^{\pi}_{0} \PP\left( |g(x)|\leq s(x) \right) \, dx = \int_{J_1 \cup J_2} 1 \, dx  + \Theta\left( \frac{1}{m^{1/2}}\right)\int_{2\pi/m}^{\pi} x^{-1}\, dx
 = \Theta\left( \frac{\log m}{m^{1/2}} \right), \] as desired.\end{proof}

\vspace{4mm}

One can see that Lemma~\ref{lem:time-in-envelope} along with Lemma~\ref{lem:zeros-and-measure} implies the upper bound on $\EE\, Z(f)$, in Theorem~\ref{thm:main2}.
For the lower bound we need one further probabilistic calculation.

\begin{lemma}\label{lem:sign-change} For $m \leq n$ and $j \in \{m/2,\ldots,m-1\}$, let $f \sim \cF_{m,n}$, $j \in \{0,\ldots, m-1\}$, 
and let $I = [\pi j/m,\pi (j+1)/m ]$. Then
\[ \PP( Z_I(f) \geq 1 ) \geq 1/4 +o_m(1).\]
\end{lemma}

\begin{proof}[Proof of Lemma~\ref{lem:sign-change}]
Set $I = [\pi j/m,\pi (j+1)/m ] =: [a,b]$. We note that $|D_n(x)| < 10$, for $x>\pi/4$, so the event ``$g(a) > 10$ and $g(b)< -10$'' is enough to guarantee a root of $f$. For this, we look to apply Theorem~\ref{thm:BE} to the $\R^2$-valued random variable 
\[ X = (g(a),g(b)) = \sum_{k=0}^m \eps_k(\cos ka,\cos kb ) \] and the convex set $C :=  \{ (x,y) :x > 10, y < -10\}$.
Note that
\begin{equation} \Cov(g(a),g(b)) = \sum_{k=0}^m \cos(ka)\cos(kb) = 0, \end{equation}
and 
$$\beta_2(a,b) := \sum_{k=0}^m \E \|(\varepsilon_{k}-1/2)\cos(ka),(\varepsilon_{k}-1/2)\cos(kb))\|_{2}^{\frac{3}{2}}\leq m.$$
Now let $Z = (Z(a),Z(b)) \sim N (\mu, \Sigma)$ where $\mu = (\EE\, g(a), \EE\, g(b))$ and 
\[ \Sigma = \begin{pmatrix}
\s^2(a) & 0 \\  0 & \s^2(b)  
\end{pmatrix}.\]
So we may apply Theorem~\ref{thm:BE} along with our lower bound on the standard deviation at \eqref{eq:s} to learn that
\[\left| \PP\left( X \in C \right) - \PP( Z \in C) \right| \leq \min\{ \s(a)^{-3},\s(b)^{-3}\}\beta_2(a,b) = O(m^{1/2}) \] 
and therefore
\[\PP\left( X \in C \right) = \PP( Z(a) \geq 10)\PP( Z(b) \leq -10) + O(m^{-1/2}) = 1/4 + o_{m}(1). \]
The last line is justified by the fact that $Z(a),Z(b)$ are independent normal random variables with $\EE\, g(a), \EE\, g(b) = O(1)$ and $\s(a),\s(b) \geq \Omega(m^{1/2})$.\end{proof}

\section{Finishing the upper bound for $m > n^{1-\alpha}$}\label{sec:large-m}

Our Theorem~\ref{thm:ET-for-special-polys} is insufficient to prove our upper bound on $\E\, Z(f)$, when $m = n^{1-o(1)}$
and in this section we turn to deal with these large values of $m$. The main result of this section is the following.

\begin{lemma} \label{lem:large-m-E-bound}
Let $m \in  [n^{0.99},n]$. For $f \sim \cF_{m,n}$ we have
\[ \EE\, Z(f) = O(m). \]
\end{lemma}
To prove this Lemma, we introduce a set related to $\cE(g)$,
\begin{equation}\label{eq:defE+} \cE'_n(g) := \{ x : |g'(x)| \leq 2^7 n s(x) \}.\end{equation}
The following lemma shows the utility of this definition: if a short interval $I$ is not entirely contained in $\cE_n'(g)$ then $f$ does not have too many roots in $I$. We point out that Lemma~\ref{lem:bigm-short-ints} is a sort of companion to Lemma~\ref{lem:two-roots-reux} and, indeed, a similar idea is applied. Lemma~\ref{lem:bigm-short-ints} will be applied with $\delta = n^{-0.1}$ and so this lemma is telling us that $Z_I(f) = O(1)$. 

\begin{lemma}\label{lem:bigm-short-ints}
For $\delta \in (0,1/4)$, let $I \subseteq [2^{15}(\delta n)^{-1},\pi]$ be an interval of length $|I| = \delta/n$,
let $g$ be a $\{0,1\}$-cosine polynomial with $\deg(g) \leq n$ and let $f = D_n - g$. If $I \not\subset \cE'_n(g)$ then
\[ Z_I(f) \leq \frac{\log 3n}{\log 1/\delta } + 1. \] 
\end{lemma}
\begin{proof}
Let $x_0 \in I$ be a value where $|g'(x_0)| \geq (2^7n)s(x_0)$ and let $t+1 = \min\{ Z_I(f),\log n\}$.
Now, since $f$ has at least $t+1$ zeros in $I$, $f'$ has at least $t$ zeros in $I$, so we may apply Lemma~\ref{lem:mean-value} to $f'$ to get
\begin{equation} \label{eq:bigmmain} \frac{n^t}{\delta^t} \left( |g'(x_0)| - |D'_n(x_0)| \right) \leq \frac{n^t}{\delta^t}|f'(x_0)| \leq  \max_{x \in I } |f^{(t+1)}(x)|. \end{equation}
Since $x_0 \in [ 2^{15}\delta^{-1}/n, \pi] $ we may apply Lemma~\ref{lem:d-kernel} to learn that $|D'_n(x_0)| \leq 2^6/x_0$.
Thus, using the definition of $x_0$, we have
\begin{equation} \label{eq:g'D} |g'(x_0)| \geq 2^7 n s(x_0) \geq 2|D'_n(x_0)|. \end{equation}
So putting \eqref{eq:bigmmain} together with \eqref{eq:g'D} gives 
\begin{equation} \label{eq:shortints-lowerbound}  2^6\frac{n^{t+1}}{\delta^{t}} \leq  \max_{x \in I } |f^{(t+1)}(x)|. \end{equation}
On the other hand, we have 
\begin{equation}\label{eq:fgn} \max_{x \in I } |f^{(t+1)}(x)| \leq \max_{x \in I} |g^{(t+1)}(x)| + 2^7n^{t+1}/x_0. \end{equation}
Crudely applying the triangle inequality, gives $\max_{x \in I} |g^{(t+1)}(x)| \leq n^{t+2}$ and thus, putting \eqref{eq:fgn} together with \eqref{eq:shortints-lowerbound} gives
\[ 2^6\frac{n^{t+1}}{\delta^{t}}  \leq n^{t+2} +  2^7n^{t+1}/x_0 \leq n^{t+2}(2^7+1).\]
Rearranging gives, $\delta^{-t} \leq 3n$ and so if $t = \log n$ this automatically gives a contradiction. Thus $t+1 = Z(f)$
and so we arrive at the desired inequality.
\end{proof}

\vspace{4mm}

We shall also need the following observation.

\begin{observation}\label{lem:E'-E-transfer} For $\delta  \in (0,1)$, let $I \subseteq [4/n,\pi]$ be an interval of length $|I| = \delta/n$ and let $g$ be a differentiable function on $[0,\pi]$.
If $I \subseteq \cE'_n(g)$ and $I \cap \cE(g)\not= \emptyset$ then 
\begin{equation}\label{eq:E+def}
 I \subseteq \cE^{+}(g) := \{x : |g(x) - 1/2| \leq 4s(x) \}. 
\end{equation}
\end{observation}
\begin{proof}
Let $x \in I$ and $x_0 \in \cE(g)\cap I$. By the mean value theorem, there exists $x_1 \in I$ for which
\[ |g(x)-1/2| = \left| g(x_0) -1/2 + g'(x_1)(x_0-x_1)\right| \leq \left| g(x_0)-1/2 \right| + ns(x_1)(\delta/n) \]
which is at most $s(x_0) + s(x_1) \leq 4s(x)$. \end{proof}

\vspace{4mm}

The following Lemma will later be applied to show that $|\cE(g)\cap \cE_n'(g)|$ is small. 

\begin{lemma}\label{lem:Hal-app} For $m \in \N$ and $x$ with $d(x,\pi \Z) \geq 2^{8}/m$, 
let $B \subseteq \R^2$ be an open ball of diameter $\Delta = 2^{-5}$. If $g \sim \cG_m$ we have
\[ \PP\left( (|g(x)|, |m^{-1}g'(x)|) \in B \right) = O(m^{-1}). \]
\end{lemma}

To prove Lemma~\ref{lem:Hal-app}, we apply the following Theorem of Hal\'asz~(\cite{halasz}, Theorem 1), which has been rephrased and simplified slightly for our purposes.

\begin{theorem}\label{thm:Hal}
For $\Delta, \delta >0$, let $a_1,\ldots,a_n \in \R^{2}$ be such that for any $v \in \R^2$, with $|v|_2 = 1$,  $|\langle  a_k, v\rangle|\geq \Delta$ for at least $\delta n$ of the vectors $\{a_k\}$. Then for any open ball $B \subseteq \R^2$ of diameter at most $\Delta$, we have 
$$\mathbb{P}(a_{1}\varepsilon_{1}+\cdots+a_{n}\varepsilon_{n}\in B) =O_{\delta}(n^{-1}),$$
where $\eps_1,\ldots,\eps_m$ are iid Bernoulli random variables with $\PP(\eps_k=1) = \PP(\eps_k=0) = 1/2$. \end{theorem}

We now turn to the proof of Lemma~\ref{lem:Hal-app}.

\begin{proof}[Proof of Lemma~\ref{lem:Hal-app}] We show that the conditions of Hal\'asz's Theorem are satisfied for
$a_{k}:=(\cos(kt),-(k/m)\sin(kt))$, where $k \in \{1,\ldots, m\}$, $\delta= 2^{-4}$ and $\Delta = 2^{-5}$. Let $v=(v_1,v_2)$ be a unit vector and consider the average inner product $|\la v, a_k\ra|^2$;
\begin{eqnarray*}
\sum_{k=1}^{m}|\langle a_k,v\rangle|^2&=&  v_{1}^2 \sum_{k=1}^{m} \cos^{2}(kx)+(v_{2}/m)^{2}\sum_{k=1}^{m}(k^2\sin^{2}(kx))\\
&=& v_{1}^{2}/2\sum_{k=1}^{m}(1+\cos(2kx))+ (v_{2}/m)^{2}/2 \sum_{k=1}^{m}k^2(1-\cos(2kx))\\
&\geq& v_{1}^{2}/2(m+D_{m}(2x))+(1/2)(v_{2}/m)^{2}(m^{3}/6 +1/8D_{m}^{(2)}(2x)). \end{eqnarray*}
Now, by Lemma~\ref{lem:d-kernel}, we see that for $x$ satisfying $d(x,\pi \Z) \geq 2^{17}/m$, we have $D_{m}^{(2)}(2x) \leq 2^8m^2/x \leq 2^{-9}m$.
Therefore, 
\[ \frac{1}{m} \sum_{k=1}^{m}|\langle a_k,v\rangle|^2 \geq 2^{-4} v_{1}^{2} + 2^{-4}v_{2}^{2} = 2^{-4}. \]
Since $\la a_k, v\ra \leq \|a_k\|\|v\| \leq 1$, there must be at least $2^{-4}$ vectors with $\la a_k, v \ra \geq 2^{-5}$.
We therefore may apply Theorem~\ref{thm:Hal} to finish the proof of the \ref{lem:Hal-app}.
\end{proof}

\begin{lemma}\label{lem:measure-E+} For $m \leq n$, we have 
\[ \EE|\cE'_n(g) \cap \cE^+(g) \cap [n^{-0.1},\pi]| = O\left( \frac{n^{1.1}}{m^2}\right). \]
\end{lemma}
\begin{proof}
We proceed as we did in Lemma~\ref{lem:time-in-envelope}, except we use Hal\'{a}sz's theorem (Theorem~\ref{thm:Hal}) instead of the Berry-Essen-type theorem, Theorem \ref{thm:BE}.
Note that it is enough to only consider $x$ for which $d(x,\{0,\pi\}) \geq C/m$ as our final bound $n^{1.1}/m^2$ is larger than $1/m$. We want to obtain an upper bound on the quantity
\[ \PP\left( g(x) \leq 1/x ,\,  g'(x) \leq 2^7n/x \right) = \PP\left( g(x) \leq 1/x ,\,  m^{-1}g'(x) \leq 2^7n/(mx) \right). \]
To do this, note that we can cover the box $\{ v \in \R^2 : |v_1| \leq 1/x, |v_2| \leq 2^7n/(mx) \}$ with $t = O(n/(mx^2))$ translates $B_1,\ldots,B_t$ of
$ \{ v \in \R^2 : |v_1| \leq 2^{-3}, |v_2| \leq 2^{-3}\}$, which has diameter $<2^{-5}$.
Thus we have
\[ \PP\left( g(x) \leq 1/x ,\,  m^{-1}g'(x) \leq 2^7n/(mx) \right) = \sum_{i=1}^t \PP\left ((g(x),m^{-1}g'(x)) \in B_i \right) = O\left(\frac{n}{(mx)^2}\right), \]
where we have applied Lemma~\ref{lem:Hal-app} to each summand. 
To finish, we simply note that 
\begin{eqnarray*} \EE |\cE'_n(g) \cap \cE^+(g) \cap [n^{-0.1},\pi]| &=& \int_{n^{-0.1}}^{\pi} \PP\left( g(x) \leq 1/x ,\,  g'(x) \leq 2^7n/x \right)\, dx  \\
&\leq& \frac{Cn}{m^2}\int_{n^{-0.1}}^{\pi-2^8/m} x^{-2} \, dx = O\left( \frac{n^{1.1}}{m^2} \right), \end{eqnarray*} as desired.\end{proof}

\vspace{4mm}

We now turn to prove Lemma~\ref{lem:large-m-E-bound}, the main objective of this section.

\begin{proof}[Proof of Lemma~\ref{lem:large-m-E-bound}]
We first notice that Theorem~\ref{thm:ET} tells us that 
\[ Z_{[0,n^{-0.1}]}(f) = O(n^{-0.9}) = O(m) \] and so we may assume that $x \geq n^{-0.1}$ in what follows.

We set $\delta = n^{-0.1}$ and let $\cI$ be a partition of $[n^{-0.1},\pi]$ into intervals of length $\delta/n$. 
From Observation~\ref{obs:danger}, we know that if $I$ contains a root then we must have $I \cap \cE(g) \not= \emptyset$. We call such an interval \emph{dangerous}. 
To count the number of dangerous intervals we notice they come in two types;
call a dangerous interval an \emph{interior} interval if $I \subseteq \cE(g)$ and call a dangerous interval $I$ a \emph{boundary} interval if $I \not\subseteq \cE(g)$.
Lemma~\ref{lem:few-crossings} tells us that there are at most $O(m)$ boundary intervals. On the other hand, if we let $a(f)$ be the number of interior intervals
we see that $a(f)(\delta/n) \leq |\cE(g)|$ and so, applying Lemma~\ref{lem:time-in-envelope}, we have
\[ \EE\, a(f)   \leq n \delta^{-1} \EE |\cE(g)| = O\left( \frac{n \log m}{\delta m^{1/2}}\right ) = O(m), \]
where the last inequality follows from the fact that $m > n^{0.9}$ and $\delta  = n^{-0.1}$.
Putting these observations together, we see that there are at most $m$ dangerous intervals, in expectation.

We now consider two further types of dangerous intervals. We call a dangerous interval, \emph{bad} if $I$ is contained in $\cE'_n(g)$ and call an interval \emph{good} if $I$ is not contained in $\cE'_n(g)$.

If $I$ is a good interval, we note that $I \subseteq [n^{-0.1},\pi]$ and so we can apply Lemma~\ref{lem:bigm-short-ints} to see that $Z_I(f) = O(1)$ and therefore the expected number of roots in good intervals is at most $O(m)$.

We now count the number of roots in bad intervals. 
By Lemma~\ref{lem:E'-E-transfer}, we see that if $I \subseteq [n^{-0.1},\pi]$ is bad then\footnote{Recall that $\cE^+(g)$ is defined at \eqref{eq:E+def} 
and is a slightly enlarged version of $\cE(g)$.} $ I\subseteq \cE_n'(g)\cap \cE^{+}(g) \cap [n^{-0.1},\pi ]$.
So, if we let $b(f)$ be the number of bad intervals in $\cI$, we have 
\[ b(f) (\delta/n) \leq  |\cE_n'(g)\cap \cE^{+}(g) \cap [n^{-0.1},\pi ]|. \]
So taking expectations and applying Lemma~\ref{lem:measure-E+} yields
\begin{equation} \label{eq:Eb(f)} \EE\, b(f) \leq  (n/\delta) \EE\, |\cE_n'(g)\cap \cE^{+}(g) \cap [n^{-0.1},\pi] | =O\left(\frac{n^{2.1}}{\delta m^2}\right), \end{equation}
where we have applied Lemma~\ref{lem:measure-E+}. 
We can apply Theorem~\ref{thm:ET} to each of these bad intervals to conclude that the number of zeros that $f$ has in bad intervals is at most
\[  n|\cE_n'(g)\cap \cE^{+}(g)| + C b(f) (n\log n)^{1/2},  \]
where $C>0$ is an absolute constant. 
And therefore, using Lemma~\ref{lem:measure-E+} and \eqref{eq:Eb(f)}, we have that the expected number of roots in bad intervals is at most
\[ n \EE\, |\cE_n'(g)\cap \cE^{+}(g)| + C(n\log n)^{1/2} \EE\, b(f)  = O\left(\frac{n^{2.61}}{\delta m^2}\right), \]
which is $O(m)$, since $n^{2.61}\delta^{-1}/m^2 \leq n^{2.71}/n^{1.8}  = n^{0.91} \leq m$. This finishes the proof. 
\end{proof}

\section{Proofs of main theorems}
All that remains is to put the pieces together and prove Theorem~\ref{thm:main1} and Theorem~\ref{thm:main2}.

\begin{proof}[Proof of Theorem~\ref{thm:main2}]

We begin with the proof of the following lower bound. If we put
\[I_i := [ \pi i/m ,  \pi (i+1)/m ],\] for $m/2 \leq i \leq m-1$, Lemma~\ref{lem:sign-change} tells us that 
\begin{equation} \label{eq:final1} \EE\, Z(f) \geq \sum_{i=m/2}^m \PP( Z_{I_i}(f) \geq 1) \geq m/8 +o(m). \end{equation}
Now let us consider the case $m < n^{0.99}$. With this assumption in hand, we may apply Lemma~\ref{lem:zeros-and-measure}
and then take expectations to see that 
\begin{equation}\label{eq:final1.5} \EE\, Z(f) \leq C\left(n \EE\, |\cE(g)| + m + n^{0.6}\right)= O\left( \frac{n\log m}{m^{1/2}} + m \right). \end{equation}
Likewise, for the lower bound, we have
\begin{equation}\label{eq:final2} \EE\, Z(f) = \Omega\left( \frac{n\log m}{m^{1/2}} - m \right). \end{equation}
Thus, \eqref{eq:final1.5}, along with an appropriate convex combination of the inequalities \eqref{eq:final1},\eqref{eq:final2}, yields Theorem~\ref{thm:main2}.

In the case that $m \geq n^{0.99}$, the upper bound in Theorem~\ref{thm:main2} follows from Lemma~\ref{lem:large-m-E-bound}, while \eqref{eq:final1} furnishes a matching 
lower bound. 
\end{proof}

\vspace{4mm}

There is a \emph{tiny} hiccup in the direct derivation of Theorem~\ref{thm:main1} from Theorem~\ref{thm:main2} as we do not explicitly control the number of terms in the resulting polynomial and therefore we cannot claim the theorem for \emph{all} sizes of $|A|$. To get around this, we instead derive Theorem~\ref{thm:main2} from Lemma~\ref{lem:zeros-and-measure}.

\begin{proof}[Proof of Theorem~\ref{thm:main1}]
Note that it is enough to prove Theorem~\ref{thm:main1} for all sufficiently large $N := |A|$. Now, given $N$ sufficiently large, put $m = (N\log N)^{2/3}$.
By Lemma~\ref{lem:time-in-envelope}, we know that there is exists a polynomial with $0 \leq t \leq m$ terms and $ |\cE(g)| = O\left(\frac{\log m}{ m^{1/2}}\right)$.
Choose $n$ so that $N = n - t$. By Lemma~\ref{lem:zeros-and-measure} we have 
\[ Z(f) = O\left( \frac{n\log m}{m^{1/2}} + t \right) \leq C(N\log N)^{3/2}, \] for an absolute constant $C >0$.\end{proof}

\section{acknowledgments}
\noindent We thank B\'{e}la Bollob\'{a}s and Rob Morris for comments.

\bibliographystyle{abbrv}
\bibliography{CosinesBib}

\newcommand{\SortNoop}[1]{}
\begin{thebibliography}{10}

\bibitem{amoroso}
F.~Amoroso and M.~Mignotte.
\newblock On the distribution on the roots of polynomials.
\newblock In {\em Ann. de l'institut Fourier}, volume~46, pages 1275--1291,
  1996.

\bibitem{blatt}
H.-P. Blatt.
\newblock On the distribution of simple zeros of polynomials.
\newblock {\em J. of approx. theory}, 69(3):250--268, 1992.

\bibitem{bloch-polya}
A.~Bloch and G.~P{\'o}lya.
\newblock On the roots of certain algebraic equations.
\newblock {\em Proc. of the London Math. Soc.}, 2(1):102--114, 1932.

\bibitem{borweinQs}
P.~Borwein and T.~Erd{\'e}lyi.
\newblock Questions about polynomials with $\{$0, -1, +1$\}$ coefficients.
\newblock {\em Constr. Approx.}, 12:439--442, 1996.

\bibitem{BElowerbounds}
P.~Borwein and T.~Erd\'{e}lyi.
\newblock Lower bounds for the number of zeros of cosine polynomials in the
  period: a problem of {L}ittlewood.
\newblock {\em Acta Arith.}, 128:377--384, 2007.

\bibitem{BEFL}
P.~Borwein, T.~Erd\'{e}lyi, R.~Ferguson, and R.~Lockhart.
\newblock On the zeros of cosine polynomials: solution to an old problem of
  {L}ittlewood.
\newblock {\em Ann. of Math.}, 167:1109--1117, 2008.

\bibitem{boyd}
D.~W. Boyd.
\newblock On a problem of {B}yrnes concerning polynomials with restricted
  coefficients.
\newblock {\em Math. of Computation}, 66(220):1697--1703, 1997.

\bibitem{erd-TE-oneside}
T.~Erd{\'e}lyi.
\newblock An improvement of the {E}rd{\H{o}}s--{T}ur{\'a}n theorem on the
  distribution of zeros of polynomials.
\newblock {\em Comptes Rendus Mathematique}, 346(5-6):267--270, 2008.

\bibitem{TEFit}
T.~Erd{\'e}lyi.
\newblock The number of unimodular zeros of self-reciprocal polynomials with
  coefficients in a finite set.
\newblock {\em Acta Arith.}, 176:177--200, 2016.

\bibitem{trsh}
T.~Erd{\'e}lyi.
\newblock Improved lower bound for the number of unimodular zeros of
  self-reciprocal polynomials with coefficients in a finite set.
\newblock {\em Acta Arith.}, 02 2017.

\bibitem{ErdosTuran}
P.~Erd\H{o}s and P.~Tur\'{a}n.
\newblock On the distribution of roots of polynomials.
\newblock {\em Ann. of Math.}, 57:105--119, 1950.

\bibitem{halasz}
G.~Hal{\'a}sz.
\newblock Estimates for the concentration function of combinatorial number
  theory and probability.
\newblock {\em Periodica Mathematica Hungarica}, 8(3-4):197--211, 1977.

\bibitem{HardyLittlewood}
G.~Hardy and J.~Littlewood.
\newblock A new proof of a theorem on rearrangements.
\newblock {\em J. London Math. Soc.}, 23:163--168, 1948.

\bibitem{JELw1}
J.~Littlewood.
\newblock On the mean values of certain trigonometrical polynomials.
\newblock {\em J. London Math. Soc.}, 36:307--334, 1961.

\bibitem{JELw2}
J.~Littlewood.
\newblock On the real roots of real trigonometrical polynomials {(II)}.
\newblock {\em J. London Math. Soc.}, 36:511--552, 1964.

\bibitem{JELw3}
J.~Littlewood.
\newblock On polynomials $\sum \pm z^m$, and $\sum e^{\alpha_n i}z^m, z =
  e^{i\theta},$.
\newblock {\em J. London Math. Soc.}, 41:367--376, 1966.

\bibitem{JEL66}
J.~Littlewood.
\newblock The real zeros and value distributions of real trigonometrical
  polynomials.
\newblock {\em J. London Math. Soc.}, 1(1):336--342, 1966.

\bibitem{JELw4}
J.~Littlewood.
\newblock {\em Some Problems in Real and Complex Analysis}.
\newblock Heath Mathematical Monographs, Lexington, Massachusetts, 1968.

\bibitem{LO38}
J.~Littlewood and A.~Offord.
\newblock On the number of real roots of a random algebraic equation.
\newblock {\em J. London Math. Soc.}, 13:288--295, 1938.

\bibitem{JEL62}
J.~E. Littlewood.
\newblock On the mean values of certain trigonometrical polynomials ii.
\newblock {\em Illinois J. Math.}, 6(1):1--39, 03 1962.

\bibitem{LO45}
J.~E. Littlewood and A.~C. Offord.
\newblock On the distribution of the zeros and $\alpha$-values of a random
  integral function (i).
\newblock {\em J. London Math. Soc.}, 1(3):130--136, 1945.

\bibitem{LO48}
J.~E. Littlewood and A.~C. Offord.
\newblock On the distribution of zeros and a-values of a random integral
  function (ii).
\newblock {\em Ann. of Math.}, pages 885--952, 1948.

\bibitem{Raic}
M.~Rai{\v{c}}.
\newblock A multivariate {B}erry--{E}sseen theorem with explicit constants.
\newblock {\em Bernoulli}, 25(4A):2824--2853, 2019.

\bibitem{Jules1}
J.~Sahasrabudhe.
\newblock Counting zeros of cosine polynomials: On a problem of {L}ittlewood.
\newblock {\em Advances in Math.}, 343:495 -- 521, 2019.

\bibitem{Schur}
I.~Schur.
\newblock Untersuchungen \"{u}ber algebraische gleichungen.
\newblock {\em Sitz. Preuss. Akad. Wiss., Phys.- Math. Kl.}, pages 403--428,
  1933.

\bibitem{steinerberger}
S.~Steinerberger.
\newblock Roots of trigonometric polynomials and the {E}rd{\H{o}}s--{T}ur{\'a}n
  theorem.
\newblock {\em Mathematika}, 66(2):245--254, 2020.

\bibitem{totik-varju}
V.~Totik and P.~P. Varj{\'u}.
\newblock Polynomials with prescribed zeros and small norm.
\newblock {\em Acta Scientiarum Mathematicarum}, 73(3-4):593--612, 2007.

\end{thebibliography}

\section{Appendix}

\begin{proof}[Proof of Lemma \ref{lem:mean-value}]
Assume that $f$ has $\ell$ roots in $I$. We may find $t_0,\ldots,t_{\ell-1} \in I$ so that $f^{(i)}(t_i) = 0 $ for each $i$. 
Now fix $x \in I$. By the Mean Value Theorem we have 
\[ f(x) = f(x) - f(t_0) = f'(s_1)(x-t_0),\]
for some $s_1 \in I$. Now, inductively applying the mean value theorem for each $i \in [\ell]$, we obtain 
\[ f^{(i)}(s_i) = f^{(i)}(s_i) - f^{(i)}(t_i) = f^{(i+1)}(s_{i+1})(s_i-t_i), \]
for some $s_{i+1} \in I$. As a result, we obtain $s_1,\ldots,s_{\ell} \in I$ so that 
\[ f(x) =  f^{(\ell)}(s_{\ell})\prod_{i=0}^{\ell-1}(s_{i} - t_i), \] 
where $s_0=x$. We then arrive at \eqref{eq:mean-value} by taking absolute values on both sides and then using $s_i,t_i\in I$ and $|I|=\eta$.\end{proof}

\vspace{4mm}

\begin{proof}[Proof of Fact \ref{fact:B}]
Define the complex valued function $f(z) := 1/\sin(z/2)$. This function is analytic in a domain not containing the points $2\pi \Z$.
So if $x \in [0,\pi]$, we may use Cauchy integral formula to write
\[ f^{(r)}(x) = \frac{r!}{2\pi i} \int_C f(z)/(z-x)^{r+1} \, dz = 
\frac{2^r r! }{2\pi x^{r}} \int_0^{2\pi} f\big( (x/2) e^{i\theta} \big)e^{-ir\t}\, d\t ,\]
where we have used the the parametrization $z = (x/2) e^{i\theta}$ for $\theta \in [0,2\pi]$. Hence 
\begin{equation}\label{ineq1}
|s^{(r)}(x)| = |f^{(r)}(x)| \leq \frac{2^r r!}{2\pi  x^{r}}\int_0^{2\pi}\left| \sin((x/2)e^{i\theta})\right|^{-1}d\theta.
\end{equation}
We now bound $|\sin(z)|$ from below in two ways. We have $\sin^{2}(x+iy)=\sin^{2}(z)+\sinh^{2}(y)$. When $|\cos(\theta)|\geq \frac{1}{2}$ we use the bound
$\left| \sin((x/2)e^{i\theta})\right|\geq |\sin(x/4)|\geq x/8$. Otherwise $|\sin(\theta)|\geq \frac{1}{2}$ and we use the bound $\left| \sin((x/2)e^{i\theta})\right|\geq |\sinh(x/4)|\geq |x|/4$ as $|\sinh(y)|\geq |y|$ for $y\in \R$.
Using these bounds we uniformly have $\left| \sin((x/2)e^{i\theta})\right|^{-1}\leq 8/x$ in (\ref{ineq1}) and obtain the desired bound.\end{proof}

\vspace{4mm}

\begin{proof}[Proof of Fact~\ref{lem:largeSIN}]
It is enough to prove the statement for $C_0(x) = \sin Tx$. By monotonicity of the derivative and periodicity of $\sin$, may assume that $0 \in I$
and therefore one of $\delta/n,-\delta/n \in I$. Using the standard inequality, $|\sin(T(\delta/2n))| \geq |T\delta/2n - (T \delta/2n)^3/6| \geq \delta/4$,
we are done.\end{proof}

\end{document}